\documentclass[10pt]{amsart}
\usepackage[psamsfonts]{amssymb}
\usepackage{amsthm,amscd}
\usepackage{type1cm}
\usepackage{color}
\newtheorem{thm}{Theorem}

\newtheorem{lem}{Lemma}
\newtheorem{cor}{Corollary}
\newtheorem{pro}{Proposition}
\theoremstyle{remark}

\theoremstyle{definition}

\allowdisplaybreaks[4]

\title[Two-dimensional endo-commutative algebras over ${\mathbb F}_2$]{A classification of two-dimensional endo-commutative algebras over ${\mathbb F}_2$}

\author[S.-E. Takahasi]{Sin-Ei Takahasi}
\author[K. Shirayanagi]{Kiyoshi Shirayanagi}
\author[M. Tsukada]{Makoto Tsukada}
\address[S.-E. Takahasi]{Laboratory of Mathematics and Games\\ Katsushika 2-371\\ Funabashi\\ Chiba 273-0032\\ Japan}
\address[K. Shirayanagi, M. Tsukada]{Department of Information Science\\ Toho University\\ Miyama 2-2-1\\ Funabashi\\ Chiba 274-8510\\ Japan}
\email{sin\_ei1@yahoo.co.jp}
\email[K. Shirayanagi (Corresponding author)]{kiyoshi.shirayanagi@is.sci.toho-u.ac.jp}
\email[M. Tsukada]{tsukada@is.sci.toho-u.ac.jp}

\dedicatory{Dedicated to Professor Yuji Kobayashi on his 77th birthday (Kiju)}

\makeatletter
\@namedef{subjclassname@2020}{\textup{2020} Mathematics Subject Classification}
\makeatother

\subjclass[2020]{Primary 17A30; Secondary 17D99, 13A99}
\keywords{Nonassociative algebras, Endo-commutative algebras, Commutative algebras, Curled algebras, Straight algebras.}

\date{\today}

\begin{document}




\maketitle

\begin{abstract}
  We introduce a new class of algebras called endo-commutative algebras in which the square mapping preserves multiplication, and provide a complete classification of endo-commutative algebras of dimension 2 over
  the field $\mathbb F_2$ of two elements.  We list all multiplication tables of the algebras up to isomorphism.  This clarifies the difference between commutativity and endo-commutativity of algebras.
\end{abstract}

\section{Introduction}\label{sec:intro}
Let $A$ be a nonassociative algebra.  The square mapping $x\mapsto x^2$ from $A$ to itself yields various important concepts of $A$.  In fact, if the square mapping of $A$ is surjective, then $A$ is said to be square-rootable (see \cite{2dim-comm-char2,z3}). Also, as is well known, if the square mapping of $A$ preserves addition, then $A$ is said to be anti-commutative.
Moreover, if the square mapping of $A$ is the zero mapping, then $A$ is said to be zeropotent.
We refer the reader to \cite{z1,z2,z3} for the details on zeropotent algebras.

The subject of this paper is another concept that also naturally arises from the square mapping.
We define $A$ to be {\it{endo-commutative}}, if the square mapping of $A$ preserves multiplication, that is, $x^2y^2=(xy)^2$ holds for all $x, y\in A$.
This terminology comes from the identity $(xx)(yy)=(xy)(xy)$ that depicts the {\it innerly} commutative property%
\footnote{The more general identity $(xy)(uv)=(xu)(yv)$ is given many other names such as {\it medial}.
For studies related to medial algebras, see \cite{JK,medial}.}.
The aim of this paper is to completely classify two-dimensional endo-commutative algebras over $\mathbb F_2$.
The strategy for the classification is based on that of \cite{2dim-comm-char2}. We can find classifications of {\it associative} algebras of dimension 2 over the real and complex number
fields in \cite{onkochishin}.
For other studies on two-dimensional algebras, see \cite{moduli,2dim,variety,classification}. 

The rest of the paper is organized as follows.  In Section~\ref{sec:iso-criterion}, we characterize two-dimensional algebras over $\mathbb F_2$ by the structure matrix with respect to a linear base whose entries are determined from the product between each pair of the base.  In the term of an equivalence relation between the matrices, we give a criterion for isomorphism between two-dimensional algebras over $\mathbb F_2$ (Proposition~\ref{pro:iso-criterion}).  By this, the problem of classifying two-dimensional algebras comes down to that of determining equivalent classes of structure matrices. In Section~\ref{sec:ec-algebras}, we characterize endo-commutativity of two-dimensional algebras over $\mathbb F_2$ in terms of structure matrices  (Proposition~\ref{pro:struc-matrix}). 

We separete two-dimensional algebras into two categories: {\it{curled}} and {\it{straight}}.  That is, a two-dimensional algebra is curled if the square of any element $x$ is a scalar multiple of $x$, otherwise it is straight.
Research related to curled algebras can be found in \cite{level2,length1}.
In Section~\ref{sec:ec-curled}, we determine endo-commutative curled algebras of dimension 2 over $\mathbb F_2$ in terms of structure matrices (Proposition~\ref{pro:ec-curled}).  In Section~\ref{sec:3special-curled}, we determine unital, commutative, and associative algebras in the family of two-dimensional curled algebras over $\mathbb F_2$ (Proposition~\ref{pro:3special-curled}). In Section~\ref{sec:class-ec-curled}, by applying the results obtained in Sections~\ref{sec:ec-curled} and \ref{sec:3special-curled}, we completely classify two-dimensional endo-commutative curled algebras over $\mathbb F_2$ into the eight algebras
\[
ECC^2_0, ECC^2_1, ECC^2_2, ECC^2_3, ECC^2_4, ECC^2_5, ECC^2_{6}\, \, {\rm{and}}\, \, ECC^2_{7}
\]
up to isomorphism (Theorem~\ref{thm:ec-curled}).  By Theorem~\ref{thm:ec-curled} and Proposition~\ref{pro:3special-curled}, we see that in the class of two-dimensional endo-commutative curled algebras over $\mathbb F_2$,
zeropotent algebras are $ECC_0^2$ and $ECC_1^2$, unital algebra is $ECC_6^2$, commutative algebras are $ECC_0^2$, $ECC_1^2$, $ECC_6^2$ and $ECC_7^2$, and associative algebras are
$ECC_0^2$, $ECC_4^2$, $ECC_5^2$ and $ECC_6^2$ up to isomorphism. Therefore, it can be said that only $ECC_2^2$ and $ECC_3^2$ are purely endo-commutative curled algebras with no special other properties.
In Section~\ref{sec:ec-straight}, we determine endo-commutative straight algebras of dimension 2 over $\mathbb F_2$ in terms of structure matrices  (Proposition~\ref{prop:ec-straight}).  In Section~\ref{sec:3special-straight},
we determine unital, commutative, and associative algebras in the family of two-dimensional straight algebras over $\mathbb F_2$ (Proposition~\ref{pro:3special-straight}).
In Section~\ref{sec:class-ec-straight}, by applying the result obtained in Sections~\ref{sec:ec-straight} and \ref{sec:3special-straight}, we completely classify two-dimensional endo-commutative straight algebras over $\mathbb F_2$ into the thirteen algebras
\[
ECS^2_1, ECS^2_2, ECS^2_3, ECS^2_4, ECS^2_5, ECS^2_{6}, ECS^2_7, ECS^2_8, ECS^2_9, ECS^2_{10}, ECS^2_{11}, \] \[ECS^2_{12}\, \, {\rm{and}}\, \, ECS^2_{13}
\]
up to isomorphism (Theorem~\ref{thm:ec-straight}).  By Theorem~\ref{thm:ec-straight} and Proposition~\ref{pro:3special-straight}, we see that in the class of two-dimensional endo-commutative straight algebras over $\mathbb F_2$,
unital algebras are $ECS_7^2$ and $ECS_{11}^2$, commutative algebras are $ECS_5^2$, $ECS_6^2$, $ECS^2_7$, $ECS_8^2$, $ECS_9^2$, $ECS_{10}^2$ and $ECS_{11}^2$, associative algebras are
$ECS_7^2, ECS_8^2$ and $ECS_{10}^2$ up to isomorphism. Therefore, it can be said that only $ECS_1^2$, $ECS_2^2$, $ECS_3^2$, $ECS^2_{4}$, $ECS^2_{12}$ and $ECS_{13}^2$ are purely endo-commutative straight algebras
with no special other properties.

Putting all this together, in the family of two-dimensional endo-commutative algebras over
$\mathbb F_2$, only eight algebras $ECC_2^2$, $ECC_3^2$, $ECS_1^2$, $ECS_2^2$, $ECS_3^2$, $ECS^2_4$, $ECS_{12}^2$ and $ECS^2_{13}$  are not zeropotent, unital, commutative nor associative. Finally, we claim that if a two-dimensional curled algebra over $\mathbb F_2$ satisfies
unitality, commutativity or associativity, then it is endo-commutative (Corollary~\ref{cor:3special-curled}).
However, this does not hold in the straight case: whereas unitality implies endo-commutativity,
if a two-dimensional straight algebra over $\mathbb F_2$ is commutative or associative, then it is not necessarily endo-commutative.


\section{A criterion for isomorphism of two-dimensional algebras}\label{sec:iso-criterion}
  For any $X=\small{\begin{pmatrix}a&b\\c&d\end{pmatrix}}\in GL_2(\mathbb F_2)$, define
\[
\widetilde{X}=\begin{pmatrix}a&b&ab&ab\\c&d&cd&cd\\ac&bd&ad&bc\\ac&bd&bc&ad\end{pmatrix}.
\]
Then we have the following:

\begin{lem}\label{lem:homo}
The mapping $X\mapsto\widetilde{X}$ is a group homomorphism from $GL_2(\mathbb F_2)$ into $GL_4(\mathbb F_2)$.
\end{lem}

\begin{proof}
Straightforward.
\end{proof}

Let $A$ be a 2-dimensional algebra over $\mathbb F_2$ with a linear base $\{e, f\}$. We write
\[
\left\{
    \begin{array}{@{\,}lll}
     e^2=a_1e+b_1f\\
     f^2=a_2e+b_2f\\
     ef=a_3e+b_3f\\
     fe=a_4e+b_4f\\
   \end{array}
  \right. 
\]
with $a_i, b_i\in \mathbb F_2\, \, (1\le i\le4)$.  Since the structure of $A$ is determined by the multiplication table $\begin{pmatrix}e^2&ef\\fe&f^2\end{pmatrix}$, we say that $A$ is the algebra on $\{e, f\}$ defined by $\begin{pmatrix}e^2&ef\\fe&f^2\end{pmatrix}$.  Also the matrix $A=\begin{pmatrix}a_1&b_1\\a_2&b_2\\a_3&b_3\\a_4&b_4\end{pmatrix}$ is called the $\it{ structure\,  matrix}$ of $A$ with respect to the base $\{e, f\}$. 

We hereafter will freely use the same symbol $A$ for the matrix and for the algebra because the algebra $A$ is determined by its structure matrix.  Then we have the following:

\begin{pro}\label{pro:iso-criterion}
Let $A$ and $A'$ be two-dimensional algebras over $\mathbb F_2$.  Then $A$ and $A'$ are isomorphic iff there is $X\in GL_2(\mathbb F_2)$ such that 
\begin{equation}\label{eq:iso-criterion}
A'=\widetilde{X^{-1}}AX.
\end{equation}
\end{pro}

\begin{proof}
Let $A$ and $A'$ be the structure matrices of $A$ on a linear base $\{e, f\}$ and $A'$ on a linear base $\{e', f'\}$, respectively.  We write
\[
A=\begin{pmatrix}a_1&b_1\\a_2&b_2\\a_3&b_3\\a_4&b_4\end{pmatrix}\, \, {\rm{and}}\, \, A'=\begin{pmatrix}c_1&d_1\\c_2&d_2\\c_3&d_3\\c_4&d_4\end{pmatrix}.
\]
Suppose that $\Phi : A\to A'$ is an isomorphism and let $X=\begin{pmatrix}a&b\\c&d\end{pmatrix}\, \, (a, b, c, d\in \mathbb F_2)$ be the matrix associated with $\Phi$, that is, $\small{\begin{pmatrix}\Phi(e)\\\Phi(f)\end{pmatrix}=X\begin{pmatrix}e'\\f'\end{pmatrix}}$, so  $X\in GL_2(\mathbb F_2)$.  By an easy calculatin, we see $\widetilde{X}A'\begin{pmatrix}e'\\f'\end{pmatrix}=AX\begin{pmatrix}e'\\f'\end{pmatrix}$, and hence we get (\ref{eq:iso-criterion}) because $\widetilde{X}^{-1}=\widetilde{X^{-1}}$ from Lemma~\ref{lem:homo}.  

Conversely, suppose that there is $X\in GL_2(\mathbb F_2)$ satisfying (\ref{eq:iso-criterion}).  Let $\Phi : A\to A'$ be the linear mapping defined by $\small{\begin{pmatrix}\Phi(e)\\\Phi(f)\end{pmatrix}=X\begin{pmatrix}e'\\f'\end{pmatrix}}$.  Then we can easily see that $\Phi$ is isomorphic by following the reverse of the above argument, hence $A$ and $A'$ are isomorphic.
\end{proof}

\begin{cor}\label{cor:rank}
  Let $A$ and $A'$ be two-dimensional algebras over $\mathbb F_2$.  If $A$ and $A'$ are isomorphic, then ${\rm{rank}} A={\rm{rank}}A'$.
\end{cor}

When (\ref{eq:iso-criterion}) holds, we say that the matrices $A$ and $A'$ are equivalent and refer to $X$ as a $\it{transformation\, \, matrix}$ for the equivalence $A\cong A'$.  Also, we call this $X$ a transformation matrix for the isomorphism $A\cong A'$ or simply for $A$ and $A'$ as well.

\section{Two-dimensional endo-commutative algebras over $\mathbb F_2$}\label{sec:ec-algebras}
Let $A$ be a two-dimensional endo-commutative algebra over $\mathbb F_2$ with the structure matrix $A=\begin{pmatrix}a_1&b_1\\a_2&b_2\\a_3&b_3\\a_4&b_4\end{pmatrix}$.  Let $x$ and $y$ be any elements of $A$ and write $\small{\left\{
    \begin{array}{@{\,}lll}
     x=x_1e+x_2f\\
     y=y_1e+y_2f.
   \end{array}
  \right. }$
Put
\[
\left\{
    \begin{array}{@{\,}lll}
     A=x_1a_1+x_2a_2+x_1x_2(a_3+a_4)\\
     B=x_1b_1+x_2b_2+x_1x_2(b_3+b_4)\\
     C=y_1a_1+y_2a_2+y_1y_2(a_3+a_4)\\
     D=y_1b_1+y_2b_2+y_1y_2(b_3+b_4)\\
     E=x_1y_1a_1+x_2y_2a_2+x_1y_2a_3+x_2y_1a_4\\
     F=x_1y_1b_1+x_2y_2b_2+x_1y_2b_3+x_2y_1b_4,
   \end{array}
  \right. 
  \]
  where obviously the symbol $A$ does not denote the algebra, 
hence $x^2=Ae+Bf, y^2=Ce+Df$ and $xy=Ee+Ff$.  Then 
\begin{align*}
x^2y^2&=(ACa_1+BDa_2+ADa_3+BCa_4)e+(ACb_1+BDb_2+ADb_3+BCb_4)f
\end{align*}
and
\begin{align*}
(xy)^2&=\{Ea_1+Fa_2+EF(a_3+a_4)\}e+\{Eb_1+Fb_2+EF(b_3+b_4)\}f.
\end{align*}
Then $A$ is endo-commutative iff
\begin{equation}\label{eq:ec}
\left\{
    \begin{array}{@{\,}lll}
    ACa_1+BDa_2+ADa_3+BCa_4=Ea_1+Fa_2+EF(a_3+a_4)\\
    ACb_1+BDb_2+ADb_3+BCb_4=Eb_1+Fb_2+EF(b_3+b_4).
         \end{array}
  \right. 
\end{equation}
holds for all $x_1, x_2, y_1, y_2\in\mathbb F_2$.  Put
\begin{align*}
 X_1=x_1&y_1, X_2=x_2y_2, X_3=x_1y_2, X_4=x_2y_1, X_5=x_1x_2y_1y_2, \\
 &X_6=x_1y_1y_2, X_7=x_1x_2y_1,  X_8=x_1x_2y_2, X_9=x_2y_1y_2.
\end{align*}
\begin{lem}\label{lem:lin-indep}
The nine polynomials $X_i\, \, (1\le i\le9)$  are linearly independent over $\mathbb F_2$.
\end{lem}

\begin{proof}
 Straightforward.
\end{proof}

By an easy calculation,  we have
\begin{align*}
 ACa_1+&BDa_2+ADa_3+BCa_4\\
 &=\{a_1+a_2b_1+a_1a_3b_1+a_1a_4b_1\}X_1+\{a_1a_2+a_2b_2+a_2a_3b_2+a_2a_4b_2\}X_2\\
&\quad+\{a_1a_2+a_2b_1b_2+a_1a_3b_2+a_2a_4b_1\}X_3+\{a_1a_2+a_2b_1b_2+a_2a_3b_1+a_1a_4b_2\}X_4\\
&\quad+\{(a_1+b_3+b_4)(a_3+a_4)+a_2(b_3+b_4)\}X_5\\
&\quad+\{(a_1+a_4b_1)(a_3+a_4)+(a_2b_1+a_1a_3)(b_3+b_4)\}X_6\\
&\quad+\{(a_1+a_3b_1)(a_3+a_4)+(a_2b_1+a_1a_4)(b_3+b_4)\}X_7\\
&\quad+\{(a_1a_2+a_3b_2)(a_3+a_4)+a_2(a_4+b_2)(b_3+b_4)\}X_8\\
&\quad+\{(a_1a_2+a_4b_2)(a_3+a_4)+a_2(a_3+b_2)(b_3+b_4)\}X_9
\end{align*}
and
\begin{align*}
&Ea_1+Fa_2+EF(a_3+a_4)\\
&=\{a_1+a_2b_1+a_1b_1(a_3+a_4)\}X_1+\{a_1a_2+a_2b_2+a_2b_2(a_3+a_4)\}X_2\\
&\quad+\{a_1a_3+a_2b_3+a_3b_3(a_3+a_4)\}X_3+\{a_1a_4+a_2b_4+a_4b_4(a_3+a_4)\}X_4\\
&\quad+(a_3+a_4)(a_1b_2+a_2b_1+a_3b_4+a_4b_3)X_5\\
&\quad+(a_3+a_4)(a_1b_3+a_3b_1)X_6\\
&\quad+(a_3+a_4)(a_1b_4+a_4b_1)X_7\\
&\quad+(a_3+a_4)(a_2b_3+a_3b_2)X_8\\
&\quad+(a_3+a_4)(a_2b_4+a_4b_2)X_9.
\end{align*}
Then we see  from Lemma~\ref{lem:lin-indep} that the first equation of (\ref{eq:ec}) holds for all  $x_1, x_2, y_1, y_2\in\mathbb F_2$ iff
\[
\left\{
    \begin{array}{@{\,}lll}
   a_1+a_2b_1+a_1a_3b_1+a_1a_4b_1=a_1+a_2b_1+a_1b_1(a_3+a_4)\\
a_1a_2+a_2b_2+a_2a_3b_2+a_2a_4b_2=a_1a_2+a_2b_2+a_2b_2(a_3+a_4)\\
  a_1a_2+a_2b_1b_2+a_1a_3b_2+a_2a_4b_1=a_1a_3+a_2b_3+a_3b_3(a_3+a_4)\\
  a_1a_2+a_2b_1b_2+a_2a_3b_1+a_1a_4b_2=a_1a_4+a_2b_4+a_4b_4(a_3+a_4)\\
 (a_1+b_3+b_4)(a_3+a_4)+a_2(b_3+b_4)=(a_3+a_4)(a_1b_2+a_2b_1+a_3b_4+a_4b_3)\\
 (a_1+a_4b_1)(a_3+a_4)+(a_2b_1+a_1a_3)(b_3+b_4)=(a_3+a_4)(a_1b_3+a_3b_1)\\
 (a_1+a_3b_1)(a_3+a_4)+(a_2b_1+a_1a_4)(b_3+b_4)=(a_3+a_4)(a_1b_4+a_4b_1)\\
 (a_1a_2+a_3b_2)(a_3+a_4)+a_2(a_4+b_2)(b_3+b_4)=(a_3+a_4)(a_2b_3+a_3b_2)\\
 (a_1a_2+a_4b_2)(a_3+a_4)+a_2(a_3+b_2)(b_3+b_4)=(a_3+a_4)(a_2b_4+a_4b_2).
   \end{array}
  \right. 
  \]
Note that the first two equations always hold in the above nine equations.  

Similarly, we see that the second equation of (\ref{eq:ec}) holds for all  $x_1, x_2, y_1, y_2\in\mathbb F_2$ iff
\[
\left\{
    \begin{array}{@{\,}lll}  a_1b_1+b_1b_2+a_1b_1b_3+a_1b_1b_4=a_1b_1+b_1b_2+a_1b_1(b_3+b_4)\\
a_2b_1+b_2+a_2b_2b_3+a_2b_2b_4= a_2b_1+b_2+a_2b_2(b_3+b_4)\\
a_1a_2b_1+b_1b_2+a_1b_2b_3+a_2b_1b_4=a_3b_1+b_2b_3+a_3b_3(b_3+b_4)\\
a_1a_2b_1+b_1b_2+a_2b_1b_3+a_1b_2b_4= a_4b_1+b_2b_4+a_4b_4(b_3+b_4)\\
b_1(a_3+a_4)+(b_2+a_3+a_4)(b_3+b_4)=(b_3+b_4)(a_1b_2+a_2b_1+a_3b_4+a_4b_3)\\
b_1(a_1+b_4)(a_3+a_4)+(b_1b_2+a_1b_3)(b_3+b_4)=(b_3+b_4)(a_1b_3+a_3b_1)\\
b_1(a_1+b_3)(a_3+a_4)+(b_1b_2+a_1b_4)(b_3+b_4)=(b_3+b_4)(a_1b_4+a_4b_1)\\
(a_2b_1+b_2b_3)(a_3+a_4)+(b_2+a_2b_4)(b_3+b_4)=(b_3+b_4)(a_2b_3+a_3b_2)\\
(a_2b_1+b_2b_4)(a_3+a_4)+(b_2+a_2b_3)(b_3+b_4)=(b_3+b_4)(a_2b_4+a_4b_2).
    \end{array}
  \right. 
\]
Note that the first two equations always hold in the above nine equations.  Therefore $A$ is endo-commutative iff
\begin{equation}\label{eq:ec-struc-together}
\left\{\begin{array}{@{\,}lll}   a_1a_2+a_2b_1b_2+a_1a_3b_2+a_2a_4b_1=a_1a_3+a_2b_3+a_3b_3(a_3+a_4)\cdots{\rm{(i)}}\\
  a_1a_2+a_2b_1b_2+a_2a_3b_1+a_1a_4b_2=a_1a_4+a_2b_4+a_4b_4(a_3+a_4)\cdots{\rm{(ii)}}\\
  (a_1+b_3+b_4)(a_3+a_4)+a_2(b_3+b_4)=(a_3+a_4)(a_1b_2+a_2b_1+a_3b_4+a_4b_3)\cdots{\rm{(iii)}}\\
 (a_1+a_4b_1)(a_3+a_4)+(a_2b_1+a_1a_3)(b_3+b_4)=(a_3+a_4)(a_1b_3+a_3b_1)\cdots{\rm{(iv)}}\\
 (a_1+a_3b_1)(a_3+a_4)+(a_2b_1+a_1a_4)(b_3+b_4)=(a_3+a_4)(a_1b_4+a_4b_1)\cdots{\rm{(v)}}\\
 (a_1a_2+a_3b_2)(a_3+a_4)+a_2(a_4+b_2)(b_3+b_4)=(a_3+a_4)(a_2b_3+a_3b_2)\cdots{\rm{(vi)}}\\
 (a_1a_2+a_4b_2)(a_3+a_4)+a_2(a_3+b_2)(b_3+b_4)=(a_3+a_4)(a_2b_4+a_4b_2)\cdots{\rm{(vii)}}\\
a_1a_2b_1+b_1b_2+a_1b_2b_3+a_2b_1b_4=a_3b_1+b_2b_3+a_3b_3(b_3+b_4)\cdots{\rm{(viii)}}\\
a_1a_2b_1+b_1b_2+a_2b_1b_3+a_1b_2b_4=a_4b_1+b_2b_4+a_4b_4(b_3+b_4)\cdots{\rm{(ix)}}\\b_1(a_3+a_4)+(b_2+a_3+a_4)(b_3+b_4)=(b_3+b_4)(a_1b_2+a_2b_1+a_3b_4+a_4b_3)\cdots{\rm{(x)}}\\
b_1(a_1+b_4)(a_3+a_4)+(b_1b_2+a_1b_3)(b_3+b_4)=(b_3+b_4)(a_1b_3+a_3b_1)\cdots{\rm{(xi)}}\\
b_1(a_1+b_3)(a_3+a_4)+(b_1b_2+a_1b_4)(b_3+b_4)=(b_3+b_4)(a_1b_4+a_4b_1)\cdots{\rm{(xii)}}\\
(a_2b_1+b_2b_3)(a_3+a_4)+(b_2+a_2b_4)(b_3+b_4)=(b_3+b_4)(a_2b_3+a_3b_2)\cdots{\rm{(xiii)}}\\
(a_2b_1+b_2b_4)(a_3+a_4)+(b_2+a_2b_3)(b_3+b_4)=(b_3+b_4)(a_2b_4+a_4b_2)\cdots{\rm{(xiv)}}\\.    
\end{array} \right. 
\end{equation}

(I) (i) and (ii) imply (iii) by an easy calculation.

(II) (viii) and (ix) imply (x) by an easy calculation.

(III) We see easily that (iv)$\Leftrightarrow$(v), (vi)$\Leftrightarrow$(vii), (xi)$\Leftrightarrow$(xii) and (xiii)$\Leftrightarrow$(xiv). 

 By (I), (II) and (III),  (\ref{eq:ec-struc-together}) can be rewritten as
\begin{equation}\label{eq:ec-struc}
\left\{\begin{array}{@{\,}lll}   
a_1a_2+a_2b_1b_2+a_1a_3b_2+a_2a_4b_1+a_1a_3+a_2b_3+a_3b_3+a_3a_4b_3=0\\
a_1a_2+a_2b_1b_2+a_2a_3b_1+a_1a_4b_2+a_1a_4+a_2b_4+a_3a_4b_4+a_4b_4=0\\
a_1a_3+a_1a_4+a_4b_1+a_2b_1b_3+a_2b_1b_4+a_1a_3b_4+a_3b_1+a_1a_4b_3=0\\
a_1a_2a_4+a_2a_4b_4+a_2b_2b_4+a_1a_2a_3+a_2b_2b_3+a_2a_3b_3=0\\
a_1a_2b_1+b_1b_2+a_1b_2b_3+a_2b_1b_4+a_3b_1+b_2b_3+a_3b_3+a_3b_3b_4=0\\
a_1a_2b_1+b_1b_2+a_2b_1b_3+a_1b_2b_4+a_4b_1+b_2b_4+a_4b_3b_4+a_4b_4=0\\
a_1a_3b_1+a_1a_4b_1+a_4b_1b_4+b_1b_2b_3+b_1b_2b_4+a_3b_1b_3=0\\
a_2a_3b_1+a_2a_4b_1+a_4b_2b_3+b_2b_3+b_2b_4+a_2b_4+a_2b_3+a_3b_2b_4=0.\\
\end{array} \right. 
\end{equation}

Therefore we have the following:
\begin{pro}\label{pro:struc-matrix}
Let $A$ be a two-dimensional algebra $A$ over $\mathbb F_2$ with structure matrix $\small{\begin{pmatrix}a_1&b_1\\a_2&b_2\\a_3&b_3\\a_4&b_4\end{pmatrix}}$.  Then $A$ is endo-commutative iff the eight scalars $a_1, b_1, a_2, b_2, a_3, b_3, a_4, b_4$ satisfy (\ref{eq:ec-struc}). 
\end{pro}

\section{Endo-commutative curled algebras of dimension 2}\label{sec:ec-curled}
For any $a, b, c,d, \varepsilon, \delta\in \mathbb F_2$, we denote by $C(a, b, c, d; \varepsilon, \delta)$ the two-dimensional  algebra over $ \mathbb F_2$ with linear base $\{e, f\}$ defined by $\begin{pmatrix}\varepsilon&0\\0&\delta\\a&b\\c&d\end{pmatrix}$. Then we see that any curled algebra of dimension 2 over $\mathbb F_2$ can be described by 
  $C(a_0, b_0, c_0, d_0; \varepsilon_0, \delta_0)$.  But the reverse is not necessarily true.  In fact, the algebra $C(0, 0, 0, 1; 0, 0)$ is not curled.  The following lemma gives a necessary and sufficient condition for $C(a, b, c, d;  \varepsilon, \delta)$ to be curled.

\begin{lem}\label{lem:curled}
The algebra $C(a, b, c, d;  \varepsilon, \delta)$ is curled iff $\varepsilon+a+c=\delta+b+d$ holds.
\end{lem}

\begin{proof}
Straightforward.
\end{proof}
\begin{lem}\label{lem:ec-curled}
The algebra $C(a, b, c, d;  \varepsilon, \delta)$ is endo-commutative and curled iff the six scalars $a, b, c, d, \varepsilon, \delta$ satisfy the following:
\begin{equation}\label{eq:ec-curled}
\left\{\begin{array}{@{\,}lll}   
\varepsilon+\delta+a+b+c+d=0\\
a(\varepsilon\delta+\varepsilon+b+bc)=0\\
c(\varepsilon\delta+\varepsilon+ad+d)=0\\
\varepsilon(a+c+ad+bc)=0\\
b(\varepsilon \delta+\delta+a+ad)=0\\
d(\varepsilon \delta+\delta+bc+c)=0\\
\delta(bc+b+d+ad)=0.   
\end{array} \right. 
\end{equation}
\end{lem}

\begin{proof}
Taking $a_1=\varepsilon, b_1=0, a_2=0, b_2=\delta, a_3=a, b_3=b, a_4=c$ and $b_4=d$ in (\ref{eq:ec-struc}),  we obtain the desired result from Proposition~\ref{pro:struc-matrix} and Lemma~\ref{lem:curled}.
\end{proof}
\vspace{2mm}

Put
\begin{align*}
&C_0=C(0, 0, 0, 0; 0, 0), C_1=C(0, 1, 0, 1; 0, 0), C_{1'}=C(1, 0, 1, 0; 0, 0), C_{1''}=C(1, 1, 1, 1; 0, 0),\\
& C_2=C(0, 1, 1, 0; 0, 0), C_3=C(1, 0, 0, 1; 0, 0), C_4=C(0, 1, 0, 0; 1, 0), C_5= C(1, 1, 0, 1; 1, 0),\\
&C_6=C(0, 1, 1, 1; 1, 0), C_7=C(0, 0, 0, 1; 1, 0), C_8=C(1, 1, 1, 0; 0, 1), C_9=C(1, 0, 1, 1; 0, 1),\\
&C_{10}=C(0, 0, 1, 0; 0, 1), C_{11}=C(1, 0, 0, 0; 0, 1), C_{12}=C(0, 0, 0, 0; 1, 1), C_{12'}=C(0, 1, 0,1; 1, 1),\\
&C_{12''}=C(1, 0, 1, 0; 1, 1), C_{13}= C(1, 1, 1, 1; 1, 1), C_{14}=C(0, 1, 1, 0; 1, 1), C_{15}=C(1, 0, 0, 1; 1, 1).
\end{align*}
and define
\[
\left\{\begin{array}{@{\,}lll}  
\mathcal{ECC}_{00}=\{C_0, C_1, C_{1'}, C_{1''}, C_2, C_3\}\\
\mathcal{ECC}_{10}=\{C_4, C_5, C_6, C_7\}\\
\mathcal{ECC}_{01}=\{C_8, C_9, C_{10}, C_{11}\}\\
\mathcal{ECC}_{11}=\{C_{12}, C_{12'}, C_{12''}, C_{13}, C_{14}, C_{15}\}.
\end{array} \right. 
\]
Then we have the following:
\begin{pro}\label{pro:ec-curled}
All algebras in $\mathcal {ECC}_{00}\cup\mathcal{ECC}_{10}\cup\mathcal {ECC}_{01}\cup\mathcal {ECC}_{11}$ are endo-commutative and curled.
Conversely, an arbitrary endo-commutative curled algebra of dimension 2 over $\mathbb F_2$ is isomorphic to either one of algebras in $\mathcal {ECC}_{00}\cup\mathcal{ECC}_{10}\cup\mathcal {ECC}_{01}\cup\mathcal {ECC}_{11}$.
\end{pro}

\begin{proof}
If $(\varepsilon, \delta)=(0, 0)$, then 
\begin{align*}
(\ref{eq:ec-curled})&\Leftrightarrow\left\{\begin{array}{@{\,}lll}   
a+b+c+d=0\\
ab(1+c)=0\\
cd(a+1)=0\\
ab(1+d)=0\\
cd(b+1)=0.
\end{array} \right.\\
\end{align*}
If $ab\ne0$, then $c=d=1$ by the second and fourth equations above. Assume $ab=0$. If $(a,b)=(0,1)$ or $(1,0)$, then $(c,d)=(1,0)$ or $(0,1)$ by the first equation.
If $(a,b)=(0,0)$, then $(c,d)=(0,0)$ by the first and third equations.
Therefore, we have $(\ref{eq:ec-curled})\Leftrightarrow (a, b, c, d)\in\{(0, 0, 0, 0), (0, 1, 0, 1), (1, 0, 1, 0), (1, 1, 1, 1), (0, 1, 1, 0), (1, 0, 0, 1)\}$, 
and hence any algebra in $\mathcal{ECC}_{00}$ must be endo-commutative and curled by Lemma~\ref{lem:ec-curled}.  Similarly, if $(\varepsilon, \delta)=(1, 0)$, then 
\[
(\ref{eq:ec-curled})\Leftrightarrow(a, b, c, d)\in\{(0, 1, 0, 0), (1, 1, 0, 1), (0, 1, 1, 1), (0, 0, 0, 1)\}
\]
and hence any algebra in $\mathcal{ECC}_{10}$ must be endo-commutative and curled by the same lemma.  If $(\varepsilon, \delta)=(0, 1)$, then 
\[
(\ref{eq:ec-curled})\Leftrightarrow(a, b, c, d)\in\{(1, 1, 1, 0), (1, 0, 1, 1), (0, 0, 1, 0), (1, 0, 0, 0)\}
\]
 \ 
and hence any algebra in $\mathcal{ECC}_{01}$ must be endo-commutative and curled by the same lemma.  If $(\varepsilon, \delta)=(1, 1)$, then 
 \begin{align*}
(\ref{eq:ec-curled})
&\Leftrightarrow(a, b, c, d)\in\{(0, 0, 0, 0), (0, 1, 0, 1), (1, 0, 1, 0), (1, 1, 1, 1), (0, 1, 1, 0), (1, 0, 0, 1)\},
\end{align*}
 and hence any algebra in $\mathcal{ECC}_{11}$ must be endo-commutative and curled by the same lemma.  Therefore, the first half of the proposition has been proved.

Note that any curled algebra of dimension 2 over $\mathbb F_2$ must be isomorphic to some $C(a_0, b_0, c_0, d_0; \varepsilon_0, \delta_0)$.  Moreover, if this $C(a_0, b_0, c_0, d_0; \varepsilon_0, \delta_0)$ is endo-commutative, then $a_0, b_0, c_0, d_0,\varepsilon_0$ and $\delta_0$ must satisfy (\ref{eq:ec-curled}) by Lemma~\ref{lem:ec-curled}, and hence $C(a_0, b_0, c_0, d_0; \varepsilon_0, \delta_0)$ must be in $\mathcal {ECC}_{00}\cup\mathcal{ECC}_{10}\cup\mathcal {ECC}_{10}\cup\mathcal {ECC}_{11}$ from the above four calculations.  Therefore, the second half of the proposition has been proved.
\end{proof}
\vspace{1mm}

\section{Curled algebras of dimension 2: unital, commutative and associative cases}\label{sec:3special-curled}
In this section, we determine unital, commutative, or associative curled algebras of dimension 2 over $\mathbb F_2$.


(I) Unital case

First of all, we determine unital curled algebras of dimension 2 over $\mathbb F_2$.
A curled algebra $A=C(a, b, c, d; \varepsilon, \delta)$ is unital iff
\[
\exists u\in A : ue=eu=e\, \,{\rm{and}}\, \, uf=fu=f.
\]

Put $u=\alpha e+\beta f$. Then, the above equations are rewritten as

\[
(\sharp : \alpha, \beta)
\left\{\begin{array}{@{\,}lll} \alpha\varepsilon+\beta c=\alpha\varepsilon+\beta a=\alpha b+\beta\delta=\alpha d+\beta\delta=1\\
\beta d=\beta b=\alpha a=\alpha c=0.
\end{array} \right. 
\]
Hence, $A$ is unital iff there exist $\alpha, \beta\in\mathbb{F}_2$ satisfying $(\sharp :\alpha, \beta)$.

Now let us consider two cases. When $\alpha=0$, we have $(\sharp :\alpha, \beta)\Leftrightarrow \left\{\begin{array}{@{\,}lll}\beta=c=a=\delta=1\\ d=b=0 \end{array}\right.$.
Since $A$ is curled, we have $\varepsilon+a+c=\delta+b+d$ by Lemma~\ref{lem:curled}. Hence, $\varepsilon =1$ and so $A=C(1, 0, 1, 0; 1, 1)=C_{12''}$.
When $\alpha=1$, we have $(\sharp :\alpha, \beta)\Leftrightarrow\left\{\begin{array}{@{\,}lll}\beta d=c=a=0\\b+\beta\delta=\varepsilon=1, d=b. \end{array}\right.$Since $A$ is curled, we have $\varepsilon+a+c=\delta+b+d$ by Lemma~\ref{lem:curled}. Hence, $\delta=1$ and so $A=C_{12'}$ or $C_{12}$,
depending on the value of $\beta$.
Therefore, we have the following:

\begin{lem}\label{lem:unital-curled}
Suppose the algebra $C(a, b, c, d; \varepsilon, \delta)$ is curled.  Then $C(a, b, c, d; \varepsilon, \delta)$ is unital iff it is equal to either one of $C_{12}, C_{12'}$ and $C_{12''}$.
\end{lem}

(II) Commutative case

Next, we determine commutative curled algebras of dimension 2 over $\mathbb F_2$.

Let $A=C(a, b, c, d; \varepsilon, \delta)$.
Put $\left\{\begin{array}{@{\,}lll} x=x_1e+x_2f\\y=y_1e+y_2f,\end{array} \right.$  where $x_1, x_2, y_1, y_2\in\mathbb F_2$.  Then we see easily that 
\[
xy=yx\Leftrightarrow\left\{\begin{array}{@{\,}lll} x_1y_2a+x_2y_1c=y_1x_2a+y_2x_1c\\x_1y_2b+x_2y_1d=y_1x_2b+y_2x_1d.\end{array} \right.
\]
Then we see from Lemma~\ref{lem:lin-indep} that $A$ is commutative iff $a=c$ and $b=d$.  If $A$ is commutative and curled, it follows from Lemma~\ref{lem:curled} and the above argument that $\varepsilon=\varepsilon+a+c=\delta+ b+d=\delta$.  Hence, $A$ is equal to either one of the following eight algebras: $C_0, C_1, C_{1'}, C_{1''}, C_{12}, C_{12'}, C_{12''}, C_{13}$.  Therefore, we have

\begin{lem}\label{lem:comm-curled}
Suppose the algebra $C(a, b, c, d; \varepsilon, \delta)$ is curled.  Then $C(a, b, c, d; \varepsilon, \delta)$ is commutative iff it is equal to either one of $C_0, C_1, C_{1'}, C_{1''}, C_{12}, C_{12'}, C_{12''}$ and $C_{13}$.
\end{lem}

(III) Associative case

Finally, we determine associative curled algebras of dimension 2 over $\mathbb F_2$.

Let $A=C(a, b, c, d; \varepsilon, \delta)$. 
Put $\left\{\begin{array}{@{\,}lll} x=x_1e+x_2f\\y=y_1e+y_2f\\z=z_1e+z_2f,
\end{array} \right. X_1=x_1y_1\varepsilon+x_1y_2a+x_2y_1c, X_2=x_1y_2b+x_2y_1d+x_2y_2\delta, Y_1=y_1z_1\varepsilon+y_1z_2a+y_2z_1c$ and $Y_2=y_1z_2b+y_2z_1d+y_2z_2\delta$. Then we see easily that $A$ is associative iff
\[
\left\{\begin{array}{@{\,}lll} 
X_1z_1\varepsilon+X_1z_2a+X_2z_1c=x_1Y_1\varepsilon+x_1Y_2a+x_2Y_1c\cdots(\sharp_1)\\
X_1z_2b+X_2z_1d+X_2z_2\delta=x_1Y_2b+x_2Y_1d+x_2Y_2\delta\cdots(\sharp_2)
\end{array} \right. 
\]
holds for all $x_i, y_i, z_i\in\mathbb F_2\, \, (1\le i\le2)$.  Put
\begin{align*}
&Z_1=x_1y_1z_1, Z_2=x_1y_1z_2, Z_3=x_1y_2z_1, Z_4=x_1y_2z_2,\\ &Z_5=x_2y_1z_1, Z_6=x_2y_1z_2, Z_7=x_2y_2z_1, Z_8=x_2y_2z_2.
\end{align*}
\begin{lem}\label{lem:lin-indep2}
The eight polynomials $Z_i\, \, (1\le i\le8)$ are linearly independent over $\mathbb{F}_2$.
\end{lem}

\begin{proof}
Straightforward.
\end{proof}

Note that $(\sharp_1)$ is rewritten as
\begin{align*}
&Z_1\varepsilon+Z_2\varepsilon a+Z_3(a\varepsilon+bc)+Z_4a+Z_5(c\varepsilon+dc)+Z_6ca+Z_7\delta c\\
&=Z_1\varepsilon+Z_2(a\varepsilon+ba)+Z_3(c\varepsilon+da)+Z_4\delta a+Z_5\varepsilon c+Z_6ac+Z_7c.
\end{align*}
Then we see easily from Lemma~\ref{lem:lin-indep2} that $(\sharp_1)$ holds for all $x_i, y_i, z_i\in\mathbb F_2\, \, (1\le i\le2)$ iff $ab=0, a\varepsilon+bc=c\varepsilon+ad, a=\delta a, dc=0$ and $\delta c=c$.  Similarly, note that $(\sharp_2)$ is rewritten as
\begin{align*}
&Z_2\varepsilon b+Z_3bd+Z_4(ab+b\delta)+Z_5d+Z_6(cb+d\delta)+Z_7\delta d+Z_8\delta \\
&=Z_2b+Z_3bd+Z_4\delta b+Z_5\varepsilon d+Z_6(ad+b\delta)+Z_7(cd+d\delta)+Z_8\delta.
\end{align*}
Then we see easily from Lemma~\ref{lem:lin-indep2} that $(\sharp_2)$ holds for all $x_i, y_i, z_i\in\mathbb F_2\, \, (1\le i\le2)$ iff $\varepsilon b=b, ab=0, d=\varepsilon d, cb+d\delta=ad+b\delta\, \, {\rm{and}}\, \, cd=0$.  Therefore, $A$ is associative iff
\[
(\sharp)\equiv
[ab=0, a\varepsilon+bc=c\varepsilon+ad, a=\delta a, dc=0, \delta c=c, \varepsilon b=b, d=\varepsilon d, cb+d\delta=ad+b\delta].
\]
When $\delta=0$, we see that $(\sharp)$ iff $\left\{\begin{array}{@{\,}lll}
a=c=0\\\varepsilon b=b, \varepsilon d=d.
\end{array}\right. $
On the other hand, if $A$ is curled, $\varepsilon=b+d$ since $\varepsilon+a+c=\delta+b+d$ by Lemma~\ref{lem:curled}.
Therefore, if $\varepsilon=1$, $(a, b, c, d)=(0, 0, 0, 1), (0, 1, 0, 0)$. If $\varepsilon=0$, $(a, b, c, d)=(0, 0, 0, 0)$.
Hence, in this case, $A$ is equal to either one of $\{C_0, C_4, C_7\}$.  When $\delta=1$, we see similarly that $A$ is equal to either one of $\{C_{10}, C_{11}, C_{12}, C_{12'}, C_{12''}, C_{14},  C_{15}\}$.  Hence we have

\begin{lem}\label{lem:assoc-curled}
Suppose the algebra $C(a, b, c, d; \varepsilon, \delta)$ is curled.  Then $C(a, b, c, d; \varepsilon, \delta)$ is associative iff it is equal to either one of $C_0, C_4, C_7, C_{10}, C_{11}, C_{12}, C_{12'}, C_{12''}, C_{14}$ and $C_{15}$.
\end{lem}

The following result immediately follows from Lemmas~\ref{lem:unital-curled}, \ref{lem:comm-curled} and \ref{lem:assoc-curled}.
\begin{pro}\label{pro:3special-curled}
Let $A$ be a curled algebra of dimension 2 over $\mathbb F_2$.  Then

${\rm{(i)}}$  $A$ is unital iff it is isomorphic to either one of $C_{12}, C_{12'}$ and $C_{12''}$.

${\rm{(ii)}}$ $A$ is commutative iff it is isomorphic to either one of $C_0, C_1, C_{1'}, C_{1''}, C_{12},\\ C_{12'}, C_{12''}$ and $C_{13}$.

${\rm{(iii)}}$  $A$  is associative iff it is isomorphic to either one of $C_0, C_4, C_7, C_{10}, C_{11}, C_{12},\\ C_{12'}, C_{12''}, C_{14}$ and $C_{15}$.\end{pro}

\begin{cor}\label{cor:3special-curled}
Suppose that $A$ is a two-dimensional curled algebra over $\mathbb{F}_2$.  If $A$ is unital, commutative or associative, then $A$ is necessarily endo-commutative.
\end{cor}

\section{Classification of endo-commutative curled algebras of dimension 2}\label{sec:class-ec-curled}
In this section, we classify endo-commutative curled algebras of dimension 2 over $\mathbb F_2$ by investigating isomorphism of each pair of algebras appearing in 
$\mathcal {ECC}_{00}\cup\mathcal{ECC}_{10}\cup\mathcal {ECC}_{01}\cup\mathcal {ECC}_{11}$. First of all, note that $C_0$ is not isomorphic to any of the other algebras since it is
the zero algebra.

 \begin{lem}\label{lem:iso-ec-curled1}
{\rm{(i)}} $C_1\cong C_{1'}$ and $C_1\cong C_{1''}$.

{\rm{(ii)}} $C_5\cong C_3$and $C_6\cong C_2$.

{\rm{(iii)}}  $C_8\cong C_6, C_9\cong C_5, C_{10}\cong C_4$ and $C_{11}\cong C_{7}$.

{\rm{(iv)}}  $C_{12}\cong C_{12'}, C_{12}\cong C_{12''}, C_{14}\cong C_{4}$ and $C_{15}\cong C_7$.
 \end{lem}

\begin{proof}
(i)  Let  $X_1=\begin{pmatrix}0&1\\1&0\end{pmatrix}$.  Then we see
 $\widetilde {X_1}C(1, 0, 1, 0; 0, 0)=C(0, 1, 0, 1; 0, 0)X_1$, and hence $C_1\cong C_{1'}$.  Let $X_2=\begin{pmatrix}1&0\\1&1\end{pmatrix}$.  Then we see $\widetilde{X_2}C(1, 1, 1, 1; 0, 0)=C(0, 1, 0, 1; 0, 0)X_2$, and hence $C_1\cong C_{1''}$.
\vspace{2mm}

(ii)  Let  $X_1=\begin{pmatrix}1&1\\0&1\end{pmatrix}$.  Then we see
 $\widetilde {X_1}C(1, 1, 0, 1; 1, 0)=C(1, 0, 0, 1; 0, 0)X_1$, and hence $C_3\cong C_5$.   Let  $X_2=\begin{pmatrix}0&1\\1&1\end{pmatrix}$.  Then we see
 $\widetilde {X_2}C(0, 1, 1, 1; 1, 0)=C(0, 1, 1, 0; 0, 0)X_2$, and hence $C_2\cong C_6$.  
 
 (iii) Let $X=\begin{pmatrix}0&1\\1&0\end{pmatrix}$.  Then we see that
$\widetilde XC_{8}=C_6X,  \widetilde XC_{9}=C_5X, \widetilde XC_{10}=C_4X$ and $\widetilde XC_{11}=C_7X$, and hence  $C_6\cong C_8, C_5\cong C_9, C_4\cong C_{10}$ and $C_{7}\cong C_{11}$.

(iv)  Let  $X_1=\begin{pmatrix}0&1\\1&1\end{pmatrix}$.  Then we see that $\widetilde {X_1}C_{12'}=C_{12}X_1, \widetilde {X_1}C_{14}=C_{4}X_1$ and  $\widetilde {X_1}C_{15}=C_7X_1$, hence $C_{12}\cong C_{12'}, C_4\cong C_{14}$ and $C_7\cong C_{15}$.  Let $X_2=\begin{pmatrix}1&1\\1&0\end{pmatrix}$.  Then we see $\widetilde{X_2}C(1, 0, 1, 0; 1, 1)=C(0, 0, 0, 0; 1, 1)X_2$, and hence $C_{12}\cong C_{12''}$.  
\end{proof}
\vspace{2mm}

\begin{lem}\label{lem:ECC00}
Up to isomorphism, we have $\mathcal{ECC}_{00}=\{C_0, C_1, C_2, C_3\}$.
\end{lem}

\begin{proof}
  By Lemma~\ref{lem:iso-ec-curled1} (i), it suffices to show that no two of $C_1, C_2, C_3$ are isomorphic to each other.

(i)  $C_1\ncong C_2, C_3$.  Since ${\rm{rank}}\, C_1=1$ and ${\rm{rank}}\, C_2={\rm{rank}}\, C_3=2$, it follows from Corollary~\ref{cor:rank} that $C_1\ncong C_2$ and $C_1\ncong C_3$.

(ii) $C_2\ncong C_3$.  In fact, suppose there is a transformation matrix $X=\begin{pmatrix}a&b\\c&d\end{pmatrix}$ for $C_2$ and $C_3$.  Then $\widetilde XC(1, 0, 0, 1; 0, 0)=C(0, 1, 1, 0; 0, 0)X$, which is rewritten as $[ab=cd=0, ad=c, bc=d, bc=a, ad=b]$.  If $a=1$, then $b=d=0$, hence $|X|=0$, a contradiction.  If $a=0$, then $c=0$, hence $|X|=0$, a contradiction.

By (i) and (ii), we obtain the desired result.
\end{proof}

\begin{lem}\label{lem:ECC10}
  No two algebras in $\mathcal{ECC}_{10}$ are isomorphic to each other.
\end{lem}

\begin{proof}

(i) $C_4\ncong C_5$ and $C_4\ncong C_6$.  By Lemma~\ref{lem:assoc-curled}, we see that $C_4$ is associative, but $C_5$ and $C_6$ are not.  

(ii) $C_4\ncong C_7$.  Suppose there is a transformation matrix $X=\begin{pmatrix}a&b\\c&d\end{pmatrix}$ for $C_4$ and $C_7$.  Then $\widetilde XC_7=C_4X$, which is rewritten as $[ab=b, c=cd=ac=ad=0, ac=c, bc=d]$, which implies $c=d=0$.  Then $|X|=0$, a contradiction.

(iii) $C_5\ncong C_6$.  This directly follows from Lemma~\ref{lem:iso-ec-curled1} (ii) and Lemma~\ref{lem:ECC00}.

(iv) $C_5\ncong C_7$ and $C_6\ncong C_7$.  By Lemma~\ref{lem:assoc-curled}, we see that $C_7$ is associative, but $C_5$ and $C_6$ are not.

Then we obtain the desired result from (i)$\sim$(iv).
\end{proof}

\begin{lem}\label{lem:c12-15}
No two algebras of $C_{12}, C_{13}, C_{14}$ and $C_{15}$ are isomorphic to each other.
\end{lem}

\begin{proof}
(i) $C_{12}\ncong C_{13}, C_{14}, C_{15}$.  By Lemma~\ref{lem:unital-curled}, we see that $C_{12}$ is unital, but $C_{13}, C_{14}$ and $C_{15}$ are not.  

(ii) $C_{13}\ncong C_{14}, C_{15}$.  By Lemma~\ref{lem:comm-curled}, we see that $C_{13}$ is commutative, but $C_{14}$ and  $C_{15}$ are not.

(iii) $C_{14}\ncong C_{15}$.  This directly follows from Lemma~\ref{lem:iso-ec-curled1} (iv) and Lemma~\ref{lem:ECC10}.

 Then we obtain the desired result from (i)$\sim$(iii).
\end{proof}

\begin{lem}\label{lem:noniso-ec-curled}
  {\rm{(i)}} $C_4\ncong C_1, C_2, C_3$ and $C_7\ncong C_1, C_2, C_3$.

{\rm{(ii)}} $C_{12}\ncong C_1, C_2, C_3$ and $C_{13}\ncong C_1, C_2, C_3$.

{\rm{(iii)}} $C_{12}\ncong C_4, C_7$ and $C_{13}\ncong C_{4}, C_7$.
\end{lem}

\begin{proof}
(i) By Lemma~\ref{lem:assoc-curled}, we see that $C_{4}$ and $C_7$ are associative, but $C_1, C_2, $ and $C_3$ are not.

(ii)  By Lemma~\ref{lem:unital-curled}, we see that $C_{12}$ is unital, but $C_1, C_2, $ and $C_3$ are not.  This implies  $C_{12}\ncong C_1, C_2, C_3$.  Since ${\rm{rank}}\, C_{13}=2$ and  ${\rm{rank}}\, C_{1}=1$, it follows that $C_{13}\ncong C_1$.  Moreover, by Lemma~\ref{lem:comm-curled}, we see that $C_{13}$ is commutative, but $C_2$ and $C_3$ are not. 

(iii) By Lemma~\ref{lem:comm-curled}, we see that $C_{12}$ and $C_{13}$ are commutative, but $C_4$ and $C_7$ are not.
\end{proof}

By Proposition~\ref{pro:ec-curled} and Lemmas~\ref{lem:iso-ec-curled1} to \ref{lem:noniso-ec-curled}, two-dimensional endo-commutative curled algebras over $\mathbb F_2$ are 
$\{C_0, C_1, C_2, C_3, C_4, C_7, C_{12},C_{13}\}$ up to isomorphism.
Here we put 
\[
ECC^2_0=C_0, ECC^2_1=C_1, ECC^2_2=C_2, ECC^2_3=C_3, 
\]
\[
ECC^2_4=C_4, ECC^2_5=C_7, ECC^2_{6}=C_{12}\, \, {\rm{and}}\, \, ECC^2_{7}=C_{13}.
\]

Then we have:

\begin{thm}\label{thm:ec-curled}
Up to isomorphism, two-dimensional endo-commutative curled algebras over $\mathbb F_2$ are exactly classified into the eight algebras
\[
ECC^2_0, ECC^2_1, ECC^2_2, ECC^2_3, ECC^2_4, ECC^2_5, ECC^2_{6}\, \, {\rm{and}}\, \, ECC^2_{7}
\]
with multiplication tables on a linear base $\{e, f\}$ defined by

\[
\begin{pmatrix}0&0\\0&0\end{pmatrix}, \begin{pmatrix}0&f\\f&0\end{pmatrix}, \begin{pmatrix}0&f\\e&0\end{pmatrix}, \begin{pmatrix}0&e\\f&0\end{pmatrix}, \begin{pmatrix}e&f\\0&0\end{pmatrix}, \begin{pmatrix}e&0\\f&0\end{pmatrix}, \begin{pmatrix}e&0\\0&f\end{pmatrix}\, \, {\rm{and}}\, \,\begin{pmatrix}e&e+f\\e+f&f\end{pmatrix},
\]
respectively. 
\end{thm}

The following proposition describes the details of Corollary~\ref{cor:3special-curled}, except for (i).

\begin{pro}\label{pro:detail-cor2}
Let $A$ be a curled algebra of dimension 2 over $\mathbb F_2$.  Then

${\rm{(i)}}$ $A$ is zeropotent iff it is isomorphic to either one of $ECC_0^2$ and $ECC_1^2$.

${\rm{(ii)}}$ $A$ is unital iff it is isomorphic to $ECC_6^2$.

${\rm{(iii)}}$ $A$ is commutative iff it is isomorphic to either one of $ECC_0^2, ECC_1^2, ECC_6^2$ and $ECC_7^2$.

${\rm{(iv)}}$  $A$ is associative iff it is isomorphic to either one of $ECC_0^2, ECC_4^2, ECC_5^2$ and $ECC_6^2$. 
\end{pro}

\begin{proof}
This follows from Proposition~\ref{pro:3special-curled} and Theorem~\ref{thm:ec-curled}.
\end{proof}

\section{Endo-commutative straight algebras of dimension 2}\label{sec:ec-straight}
Let $A$ be a 2-dimensional straight algebra over $\mathbb F_2$ with a linear base $\{e, f\}$.  By replacing the bases, we may assume that $e^2=f$.  Write $f^2=pe+qf, ef=ae+bf$ and $fe=ce+df$, hence the structure matrix of $A$ is 
\[
S(p, q, a, b, c, d)\equiv\begin{pmatrix}0&1\\p&q\\a&b\\c&d\end{pmatrix}, 
\]
where $a, b, c, d, p, q\in\mathbb F_2$.  Of course, the algebra $S(p, q, a, b, c, d)$ is always straight.  Let $a_1=0, b_1=1, a_2=p, b_2=q, a_3=a, b_3=b, a_4=c$ and $b_4=d$.  In this case,
(\ref{eq:ec-struc}) can be rewritten as
\begin{equation}\label{eq:ec-straight}
\left\{\begin{array}{@{\,}lll}   
pq+pc+pb+ab+abc=0\cdots{\rm{(i)}}\\
pq+pa+pd+acd+cd=0\cdots{\rm{(ii)}}\\
c+pb+pd+a=0\cdots{\rm{(iii)}}\\
pcd+pqd+pqb+pab=0\cdots{\rm{(iv)}}\\
q+pd+a+qb+ab+abd=0\cdots{\rm{(v)}}\\
q+pb+c+qd+bcd+cd=0\cdots{\rm{(vi)}}\\
cd+qb+qd+ab=0\cdots{\rm{(vii)}}\\
pa+pc+qbc+qb+qd+pd+pb+qad=0.\cdots{\rm{(viii)}}\\
\end{array} \right. 
\end{equation}
Therefore, by Proposition~\ref{pro:struc-matrix}, we have the following:

\begin{lem}\label{lem:ec-straight}
Suppose $p, q, a, b, c, d\in\mathbb F_2$.  Then the algebra $S(p, q, a, b, c, d)$ is endo-commutative iff the scalars $p, q, a, b, c, d$ satisfy (\ref{eq:ec-straight}). 
\end{lem}

(A) $a=c, b+d=1$: In this case, (iii) implies $p=0$.  Also (vii) implies $a=q$.  Therefore we see easily that (\ref{eq:ec-straight}) is rewritten as $\left\{\begin{array}{@{\,}lll}   
p=0\\
a=q.
\end{array} \right. $
Then the solutions $(p, q, a, b, c, d)$ of (\ref{eq:ec-straight}) are $(0, 0, 0, 0, 0, 1), (0, 0, 0, 1, 0, 0), (0, 1, 1, 1, 1, 0), (0, 1, 1, 0, 1, 1)$.

(B) $a=c, b+d=0$: We see easily that (\ref{eq:ec-straight}) is rewritten as $\left\{\begin{array}{@{\,}lll}   
pq+pa+pb=0\\
q+pb+a+qb=0.
\end{array} \right. $ If $p=1$, then (\ref{eq:ec-straight})$\Leftrightarrow\left\{\begin{array}{@{\,}lll}   
q+a+b=0\\
qb=0.
\end{array} \right.$  Also if $p=0$, then  (\ref{eq:ec-straight})$\Leftrightarrow q+a+qb=0$.  Then the solutions $(p, q, a, b, c, d)$ of (\ref{eq:ec-straight}) are $(1, 1, 1, 0, 1, 0), (1, 0, 0, 0, 0, 0), (1, 0, 1, 1, 1, 1),$\\ $(0, 0, 0, 0, 0, 0), (0, 0, 0, 1, 0, 1), (0, 1, 0, 1, 0, 1), (0, 1, 1, 0, 1, 0)$.

(C) $a+c=1$: In this case, (iii) implies $p(b+d)=1$, that is,  $p=b+d=1$.  Therefore, we see easily that (\ref{eq:ec-straight}) is rewritten as $\left\{\begin{array}{@{\,}lll}   
p=b+d=a+b=1\\
q=0.
\end{array} \right. $  Then the solutions $(p, q, a, b, c, d)$ of (\ref{eq:ec-straight}) are $(1, 0, 0, 1, 1, 0), (1, 0, 1, 0, 0, 1)$.  

Put
\begin{align*}
&S_1=S(0, 0, 0, 0, 0, 1), S_2=S(0, 0, 0, 1, 0, 0), S_3= S(0, 1, 1, 1, 1, 0),\\
& S_4=S(0, 1, 1, 0, 1, 1), S_5=S(1, 1, 1, 0, 1, 0), S_6=S(1, 0, 0, 0, 0, 0), \\
&S_7=S(1, 0, 1, 1, 1, 1),S_8=S(0, 0, 0, 0, 0, 0),S_9=S(0, 0, 0, 1, 0, 1), \\
&S_{10}=S(0, 1, 0, 1, 0, 1), S_{11}=S(0, 1, 1, 0, 1, 0),S_{12}=S(1, 0, 0, 1, 1, 0),\\
&S_{13}=S(1, 0, 1, 0, 0, 1).
\end{align*}
By (A), (B), (C) and Lemma~\ref{lem:ec-straight}, we have the following:

\begin{pro}\label{prop:ec-straight}
The algebra $S(p, q, a, b, c, d)$ is endo-commutative iff it is equal to either one of $S_1, S_2, S_3, S_4, S_5, S_6, S_7, S_8, S_9, S_{10}, S_{11}, S_{12}$ and $S_{13}$.
\end{pro}

\section{Straight algebras of dimension 2: unital, commutative and associative cases}\label{sec:3special-straight}
In this section, we determine unital, commutative, or associative straight algebras of dimension 2 over $\mathbb F_2$.

(I) Unital case

First of all, we determine unital straight algebras of dimension 2 over $\mathbb F_2$

A straight algebra $A=S(p, q, a, b, c, d)$ is unital iff
\[
\exists u\in A : ue=eu=e\, \,{\rm{and}}\, \, uf=fu=f.
\]
Put $u=\alpha e+\beta f$. Then, we have
\[
S(p, q, a, b, c, d)\,  {\rm{is\, \, unital}} \Leftrightarrow \exists \alpha, \beta\in\mathbb{F}_2 : 
\left\{\begin{array}{@{\,}lll} \beta=a=c=1\\d=b=\alpha=p\\d+q=1
\end{array} \right. \Leftrightarrow\left\{\begin{array}{@{\,}lll}a=c=1\\p=b=d\\q=1+p. \end{array} \right.
\]
The solutions of the last equations are $(p, q, a, b, c, d)=(0, 1, 1, 0, 1, 0), (1, 0, 1, 1, 1, 1)$.
Then we have

\begin{lem}\label{lem:unital-straight}
The algebra $S(p, q, a, b, c, d)$ is unital iff it is equal to either one of $S_{7}$ and $S_{11}$.  
\end{lem}


(II) Commutative case

Next, we determine commutative straight algebras of dimension 2 over $\mathbb F_2$.

Let $A=S(p, q, a, b, c, d)$.  Take $x, y\in A$ arbitrarily, and write $\left\{\begin{array}{@{\,}lll} x=x_1e+x_2f\\y=y_1e+y_2f,\end{array} \right.$  where $x_1, x_2, y_1, y_2\in\mathbb F_2$.  Then we see easily that 
\[
xy=yx\Leftrightarrow\left\{\begin{array}{@{\,}lll} x_1y_2a+x_2y_1c=y_1x_2a+y_2x_1c\\x_1y_2b+x_2y_1d=y_1x_2b+y_2x_1d.\end{array} \right.
\]
Then we see from Lemma~\ref{lem:lin-indep} that $A$ is commutative iff $a=c$ and $b=d$.
Hence, $A$ is commutative iff $A$ is equal to either one of the following 16 algebras:
\[
S(0, 0, 0, 0, 0, 0), S(0, 0, 0, 1, 0, 1), S(0, 0, 1, 0, 1, 0), S(0, 0, 1, 1, 1, 1),
\]
\[
S(0, 1, 0, 0, 0, 0), S(0, 1, 0, 1, 0, 1), S(0, 1, 1, 0, 1, 0), S(0, 1, 1, 1, 1, 1),
\]
\[
S(1, 0, 0, 0, 0, 0), S(1, 0, 0, 1, 0, 1), S(1, 0, 1, 0, 1, 0), S(1, 0, 1, 1, 1, 1)
\]
and
\[
S(1, 1, 0, 0, 0, 0), S(1, 1, 0, 1, 0, 1), S(1, 1, 1, 0, 1, 0), S(1, 1, 1, 1, 1, 1).
\]
Put
\begin{align*}
&S'_{1}=S(0, 0, 1, 0, 1, 0), S'_{2}=S(0, 0, 1, 1, 1, 1), S'_{3}=S(0, 1, 0, 0, 0, 0), S'_{4}=S(0, 1, 1, 1, 1, 1),\\
& S'_{5}=S(1, 0, 0, 1, 0, 1), S'_{6}=S(1, 0, 1, 0, 1, 0), S'_{7}=S(1, 1, 0, 0, 0, 0), S'_{8}=S(1, 1, 0, 1, 0, 1),\\
& S'_{9}=S(1, 1, 1, 1, 1, 1).
\end{align*}
Then we have:

\begin{lem}\label{lem:comm-straight}
The algebra $S(p, q, a, b, c, d)$ is commutative iff it is equal to either one of $S_i\, \, (5\le i\le11)$ and $S'_{i}\, \, (1\le i\le9)$.
\end{lem}

(III) Associative case

Finally, we determine associative straight algebras of dimension 2 over $\mathbb F_2$.

Let $A=S(p, q, a, b, c, d)$.  Take $x, y\in A$ arbitrarily, and write$\left\{\begin{array}{@{\,}lll} x=x_1e+x_2f\\y=y_1e+y_2f\\z=z_1e+z_2f,
\end{array} \right.X_1=x_1y_2a+x_2y_1c+x_2y_2p, X_2=x_1y_1+x_1y_2b+x_2y_1d+x_2y_2q, Y_1=y_1z_2a+y_2z_1c+y_2z_2p$ and $Y_2=y_1z_1+y_1z_2b+y_2z_1d+y_2z_2q$. Then we see easily that $A$ is associative iff
\[
\left\{\begin{array}{@{\,}lll} 
X_1z_2a+X_2z_1c+X_2z_2p=x_1Y_2a+x_2Y_1c+x_2Y_2p\cdots(\flat_1)\\
X_1z_1+X_1z_2b+X_2z_1d+X_2z_2q=x_1Y_1+x_1Y_2b+x_2Y_1d+x_2Y_2q\cdots(\flat_2)
\end{array} \right. 
\]
holds for all $x_i, y_i, z_i\in\mathbb F_2\, \, (1\le i\le2)$.  Put
\begin{align*}
 &Z_1=x_1y_1z_1, Z_2=x_1y_1z_2, Z_3=x_1y_2z_1,Z_4=x_1y_2z_2,\\
   &Z_5=x_2y_1z_1, Z_6=x_2y_1z_2, Z_7=x_2y_2z_1, Z_8=x_2y_2z_2.
\end{align*}
About $(\flat_1)$, note that
\begin{align*}
&X_1z_2a+X_2z_1c+X_2z_2p\\
&=Z_1c+Z_2p+Z_3bc+Z_4(a+bp)+Z_5dc+Z_6(ca+dp)+Z_7qc+Z_8(pa+qp)\\
\end{align*}
and
\begin{align*}
&x_1Y_2a+x_2Y_1c+x_2Y_2p\\
&=Z_1a+Z_2ba+Z_3da+Z_4qa+Z_5p+Z_6(ac+bp)+Z_7(c+dp)+Z_8(pc+qp).
\end{align*}

Therefore we see that ($\flat_1$) iff
\begin{align*}
&Z_1c+Z_2p+Z_3bc+Z_4(a+bp)+Z_5dc+Z_6(ca+dp)+Z_7qc+Z_8(pa+qp)\\
&=Z_1a+Z_2ba+Z_3da+Z_4qa+Z_5p+Z_6(ac+bp)+Z_7(c+dp)+Z_8(pc+qp).
\end{align*}
Then we see easily from Lemma~\ref{lem:lin-indep2} that $(\flat_1)$ holds for all $x_i, y_i, z_i\in\mathbb F_2\, \, (1\le i\le2)$ iff
\begin{equation}\label{eq:assoc-straight1}
\left\{\begin{array}{@{\,}lll}
    c=a, p=ba, bc=da, a+bp=qa, dc=p\\ ca+dp=ac+bp, qc=c+dp, pa+qp=pc+qp.
\end{array}\right. 
\end{equation}
About ($\flat_2$), note that  
\begin{align*}
&X_1z_1+X_1z_2b+X_2z_1d+X_2z_2q\\
&=Z_1d+Z_2q+Z_3(a+bd)+Z_4(ab+bq)+Z_5(c+d)+Z_6(cb+dq)+Z_7(p+qd)+Z_8(pb+q)\\
\end{align*}
and
\begin{align*}
&x_1Y_1+x_1Y_2b+x_2Y_1d+x_2Y_2q\\
&=Z_1b+Z_2(a+b)+Z_3(c+db)+Z_4(p+qb)+Z_5q+Z_6(ad+bq)+Z_7(cd+dq)+Z_8(pd+q).
\end{align*}
Therefore we see that ($\flat_2$) iff
\begin{align*}
&Z_1d+Z_2q+Z_3(a+bd)+Z_4(ab+bq)+Z_5(c+d)+Z_6(cb+dq)+Z_7(p+qd)+Z_8(pb+q)\\
&=Z_1b+Z_2(a+b)+Z_3(c+db)+Z_4(p+qb)+Z_5q+Z_6(ad+bq)+Z_7(cd+dq)+Z_8(pd+q).
\end{align*}
Then we see easily from Lemma~\ref{lem:lin-indep2} that  ($\flat_2$) holds for all $x_i, y_i, z_i\in\mathbb F_2\, \, (1\le i\le2)$ iff 
\begin{equation}\label{eq:assoc-straight2}
\left\{\begin{array}{@{\,}lll}
 d=b, q=a+b, a+bd=c+db, ab+bq=p+qb, c+d=q,\\
 cb+dq=ad+bq, p+qd=cd+dq, pb+q=pd+q.
\end{array}\right. 
\end{equation}
Hence, $A$ is associative iff both (\ref{eq:assoc-straight1}) and (\ref{eq:assoc-straight2}) hold.
It follows from this that:
\begin{equation}\label{eq:assoc-straight}
   c=a, d=b, p=ab\, \, {\rm{and}}\, \, q=a+b.
\end{equation}
By easy calculations, we see that the solutions $(p, q, a, b, c,d)$ of (\ref{eq:assoc-straight}) are $(0, 0, 0, 0, 0, 0)$, $(0, 1, 0, 1, 0, 1)$, $(0, 1, 1, 0, 1, 0)$, $(1, 0, 1, 1, 1,1)$.  

Put $S'_{10}=S(0, 1, 1, 0, 1, 0)$. Then we have the following:

\begin{lem}\label{lem:assoc-straight}
The algebra $S(p, q, a, b, c, d)$ is associative iff it is equal to either one of $S_7, S_8, S_{10}$ and $S'_{10}$.
\end{lem}

The following result immediately follows from Lemmas~\ref{lem:unital-straight}, \ref{lem:comm-straight} and \ref{lem:assoc-straight}.
\begin{pro}\label{pro:3special-straight}
Let $A$ be a straight algebra of dimension 2 over $\mathbb F_2$.  Then

${\rm{(i)}}$  $A$ is unital iff it is isomorphic to either one of $S_7$ and $S_{11}$. 

${\rm{(ii)}}$ $A$ is commutative iff it is isomorphic to either one of $S_i\, \, (5\le i\le11)$ and  $S'_i\, \, (1\le i\le 9)$.

${\rm{(iii)}}$  $A$  is associative iff it is isomorphic to either one of $S_7, S_8, S_{10}$ and $S'_{10}$.
\end{pro}
\vspace{2mm}

\section{Classification of endo-commutative straight algebras of dimension 2}\label{sec:class-ec-straight}
In this section, we classify endo-commutative straight algebras of dimension 2 over $\mathbb F_2$ up to isomorphism.

Define $\mathcal{ECS}_1=\{S_i : 1\le i\le 13, {\rm{rank}}\, S_i=1\}$ and $\mathcal{ECS}_2=\{S_i : 1\le i\le 13, {\rm{rank}}\, S_i=2\}$.  By easy observations, we have 
\[
\mathcal{ECS}_1=\{S_1, S_2, S_8, S_9, S_{10}\}\, \, {\rm{and}}\, \, \mathcal{ECS}_2=\{S_3, S_4, S_5, S_6, S_7, S_{11}, S_{12}, S_{13}\}.
\]

Since rank is an invariant of isomorphism by Corollary~\ref{cor:rank}, each algebra in $\mathcal{ECS}_1$ is not isomorphic to any algebra in $\mathcal{ECS}_2$.
Moreover, by Proposition~\ref{prop:ec-straight}, an endo-commutative straight algebra of dimension 2 over $\mathbb F_2$ is isomorphic to either one in $\mathcal{ECS}_1\cup\mathcal{ECS}_2$
and so it suffices to investigate isomorphism between the algebras in each of $\mathcal{ECS}_1$ and $\mathcal{ECS}_2$.

\begin{lem}\label{lem:es1}
No two algebras in $\mathcal{ECS}_{1}$ are isomorphic to each other.
\end{lem}

\begin{proof}
(i) $S_1$ is not isomorphic to any one of $S_2, S_8, S_9, S_{10}$.  In fact, we see from Lemma~\ref{lem:comm-straight} that $S_1$ is non-commutative, but $S_8, S_9$ and $S_{10}$ are commutative.  This implies $S_1\ncong S_8, S_9, S_{10}$. We next show $S_1\ncong S_2$.  Suppose on the contrary that $S_1\cong S_2$.  Then there is $X=\begin{pmatrix}a&b\\c&d\end{pmatrix}\in GL_2(\mathbb F_2)$ such that $\widetilde XS_2=S_1X$, which is rewritten as $\left\{\begin{array}{@{\,}lll}
   c=0\\a+ab=d\\c+cd=0\\ac+ad=0\\ac+bc=d.
\end{array}\right. 
$  This implies easily $c=d=0$, hence $|X|=0$, a contradiction.
\vspace{2mm}

(ii) $S_2$ is not isomorphic to any one of $S_8, S_9, S_{10}$.  In fact, we see from Lemma~\ref{lem:comm-straight} that $S_2$ is non-commutative, but $S_8, S_9$ and $S_{10}$ are
commutative.
\vspace{2mm}

(iii) $S_8$ is not isomorphic to any one of $S_9, S_{10}$.  In fact, we see from Lemma~\ref{lem:assoc-straight} that $S_8$ is associative, but $S_9$ is not.  This implies $S_8\ncong S_9$.  We next show $S_8\ncong S_{10}$.  Suppose on the contrary that $S_8\cong S_{10}$.  Then there is $X=\begin{pmatrix}a&b\\c&d\end{pmatrix}\in GL_2(\mathbb F_2)$ such that $\widetilde XS_{10}=S_8X$, which is rewritten as 
\[
\left\{\begin{array}{@{\,}lll}
   c=c+d=ac+bd+ad+bc=bd+bc+ad=0\\a+b=d.
\end{array}\right. 
\]
This implies easily $c=d=0$, hence $|X|=0$, a contradiction.
\vspace{2mm}

(iv) $S_9\ncong S_{10}$.  In fact, we see from Lemma~\ref{lem:assoc-straight} that $S_{10}$ is associative, but $S_9$ is not.
\vspace{2mm}

By (i), (ii), (iii) and (iv), we obtain the desired result.
\end{proof}
\vspace{2mm}

\begin{lem}\label{lem:es2}
No two algebras in $\mathcal{ECS}_{2}$ are isomorphic to each other.
\end{lem}

\begin{proof}
  (i) $S_3$ is not isomorphic to any one of $S_4, S_5, S_6, , S_7, S_{11}, S_{12}, S_{13}$.  In fact, we see from Lemma~\ref{lem:comm-straight} that $S_3$ is non-commutative, but $S_5, S_6, S_7$ and $S_{11}$ are commutative.  This implies $S_3\ncong S_5, S_6, S_7, S_{11}$.  We next show $S_3\ncong S_4$.  Suppose on the contrary that $S_3\cong S_4$.  Then there is $X=\begin{pmatrix}a&b\\c&d\end{pmatrix}\in GL_2(\mathbb F_2)$ such that $\widetilde XS_{4}=S_3X$, which implies
  $\left\{\begin{array}{@{\,}lll}
c=0\\ac+bd+bc=b+d.
 \end{array}\right. $
  This implies easily $b=d=0$, hence $|X|=0$, a contradiction. 
We next show $S_3\ncong S_{12}$.  Suppose on the contrary that $S_3\cong S_{12}$.  Then there is $X=\begin{pmatrix}a&b\\c&d\end{pmatrix}\in GL_2(\mathbb F_2)$ such that $\widetilde XS_{12}=S_3X$, which implies $\left\{\begin{array}{@{\,}lll}
b+ab=c\\a+ab=d\\d+cd=c.
\end{array}\right. $  If $a=0$, then $d=0$ by the second equation, so $c=0$ by the third equation, hence $|X|=0$, a contradiction.  If $a=1$, then $c=0$ by the first equation, so $d=0$ by the third equation, hence $|X|=0$, a contradiction.  We next show $S_3\ncong S_{13}$.  Suppose on the contrary that $S_3\cong S_{13}$.  Then there is $X=\begin{pmatrix}a&b\\c&d\end{pmatrix}\in GL_2(\mathbb F_2)$ such that $\widetilde XS_{13}=S_3X$, which implies
$\left\{\begin{array}{@{\,}lll}
b+ab=c\\a+ab=d\\d+cd=c.
\end{array}\right. $
Then we arrive at the contradiction as observed above.
\vspace{2mm}

(ii) $S_4$ is not isomorphic to any one of $S_5, S_6, S_7, S_{11}, S_{12}, S_{13}$.  In fact, we see from Lemma~\ref{lem:comm-straight} that $S_4$ is non-commutative, but $S_5, S_6, S_7$ and $S_{11}$ are commutative.  This implies $S_4\ncong S_5, S_6, S_7, S_{11}$.  We next show $S_4\ncong S_{12}$.  Suppose on the contrary that $S_4\cong S_{12}$.  Then there is $X=\begin{pmatrix}a&b\\c&d\end{pmatrix}\in GL_2(\mathbb F_2)$ such that $\widetilde XS_{12}=S_4X$, which implies $c+cd=d$.  This implies easily $d=c=0$, hence $|X|=0$, a contradiction.  We next show $S_4\ncong S_{13}$.  Suppose on the contrary that $S_4\cong S_{13}$.  Then there is $X=\begin{pmatrix}a&b\\c&d\end{pmatrix}\in GL_2(\mathbb F_2)$ such that $\widetilde XS_{13}=S_4X$, which implies $d+cd=c$.  This implies easily $c=d=0$, hence $|X|=0$, a contradiction.
\vspace{2mm}

(iii) $S_5$ is not isomorphic to any one of $S_6, S_7, S_{11}, S_{12}, S_{13}$.  In fact, we see from Lemma~\ref{lem:unital-straight} that $S_5$ is non-unital, but $S_7$ and $S_{11}$ are unital.  This implies $S_5\ncong S_7, S_{11}$.  Also we see from Lemma~\ref{lem:comm-straight} that $S_5$ is commutative, but $S_{12}$ and $S_{13}$ are non-commutative.  This implies $S_5\ncong S_{12}, S_{13}$.  We next show $S_5\ncong S_6$.  Suppose on the contrary that $S_5\cong S_6$.    Then there is $X=\begin{pmatrix}a&b\\c&d\end{pmatrix}\in GL_2(\mathbb F_2)$ such that $\widetilde XS_{6}=S_5X$, which implies $\left\{\begin{array}{@{\,}lll}
b=c=b+d\\a=d=a+c.
\end{array}\right. 
$  This implies easily $d=c=0$, hence $|X|=0$, a contradiction.
\vspace{2mm}

(iv) $S_6$ is not isomorphic to any one of $S_7, S_{11}, S_{12}, S_{13}$.  In fact, we see from Lemma~\ref{lem:unital-straight} that $S_6$ is non-unital, but $S_7$ and $S_{11}$ are unital.  This implies $S_6\ncong S_7, S_{11}$.  Also we see from Lemma~\ref{lem:comm-straight} that $S_6$ is commutative, but $S_{12}$ and $S_{13}$ are not.  This implies $S_6\ncong S_{12}, S_{13}$.
\vspace{2mm}

(v) $S_7$ is not isomorphic to any one of $S_{11}, S_{12}, S_{13}$.  In fact, we see from Lemma~\ref{lem:assoc-straight} that $S_7$ is associative, but $S_{11}, S_{12}$ and $S_{13}$
are not.  This implies $S_7\ncong S_{11}, S_{12}, S_{13}$.  
\vspace{2mm}

(vi) $S_{11}$ is not isomorphic to any one of $S_{12}, S_{13}$.  In fact, we see from Lemma~\ref{lem:comm-straight} that $S_{11}$ is commutative, but $S_{12}$ and $S_{13}$ are
not.
\vspace{2mm}

(vii) $S_{12}$ is not isomorphic to $S_{13}$.  In fact, suppose on the contrary that $S_{12}\cong S_{13}$.  Then there is $X=\begin{pmatrix}a&b\\c&d\end{pmatrix}\in GL_2(\mathbb F_2)$ such that $\widetilde XS_{13}=S_{12}X$, which implies $\left\{\begin{array}{@{\,}lll}
b+ab=c=bd+ad\\a+ab=d\\c+cd=b.
\end{array}\right. 
$ If $a=0$, then $d=0$ by the second equation, so $c=0$ by the fourth equation, hence $|X|=0$, a contradiction.  If $a=1$, then $c=0$ by the first equation, so $b=0$ by the third equation, so $d=0$ by the fourth equation, hence $|X|=0$, a contradiction. 
\vspace{2mm}

By (i), (ii), (iii), (iv), (v), (vi) and (vii),  we obtain the desired result.
\end{proof}
\vspace{2mm}

Here we put 
\[
ECS^2_1=S_1, ECS^2_2=S_2, ECS^2_3=S_3, ECS^2_4=S_4, ECS^2_5=S_5, ECS^2_{6}=S_6, ECS^2_7=S_7, 
\]
\[
ECS^2_8=S_8, ECS^2_{9}=S_9, ECS^2_{10}=S_{10}, ECS^2_{11}=S_{11}, ECS^2_{12}=S_{12} \, \, {\rm{and}}\, \, ECS^2_{13}=S_{13}.
\]
Then by Lemmas~\ref{lem:es1} and \ref{lem:es2}, we have the following:

\begin{thm}\label{thm:ec-straight}
Up to isomorphism, two-dimensional endo-commutative straight algebras over $\mathbb F_2$ are exactly classified into the thirteen algebras
\[
ECS^2_1, ECS^2_2, ECS^2_3, ECS^2_4, ECS^2_5, ECS^2_{6}, ECS^2_7, ECS^2_8, ECS^2_{9}, ECS^2_{10}, ECS^2_{11}, \] \[ECS^2_{12} \, \, {\rm{and}}\, \, ECS^2_{13}
\]
with multiplication tables on a linear base $\{e, f\}$ defined by
\[
\begin{pmatrix}f&0\\f&0\end{pmatrix}, \begin{pmatrix}f&f\\0&0\end{pmatrix}, \begin{pmatrix}f&e+f\\e&f\end{pmatrix}, \begin{pmatrix}f&e\\e+f&f\end{pmatrix}, \begin{pmatrix}f&e\\e&e+f\end{pmatrix}, 
\begin{pmatrix}f&0\\0&e\end{pmatrix}, \begin{pmatrix}f&e+f\\e+f&e\end{pmatrix}, 
\]
\[
\begin{pmatrix}f&0\\0&0\end{pmatrix}, 
\begin{pmatrix}f&f\\f&0\end{pmatrix}, 
\begin{pmatrix}f&f\\f&f\end{pmatrix}, 
\begin{pmatrix}f&e\\e&f\end{pmatrix}, \begin{pmatrix}f&f\\e&e\end{pmatrix}\, \, {\rm{and}} \, \, \begin{pmatrix}f&e\\f&e\end{pmatrix}
\]
respectively. 
\end{thm}

The following result is a restatement of Proposition~\ref{pro:3special-straight}.

\begin{pro}\label{pro:restate-3special-straight}
Let $A$ be a straight algebra of dimension 2 over $\mathbb F_2$.  Then

${\rm{(i)}}$  $A$ is unital iff it is isomorphic to either one of $ECS^2_7$ and $ECS^2_{11}$. 

${\rm{(ii)}}$ $A$ is commutative iff it is isomorphic to either one of $ECS^2_i\, \, (5\le i\le11)$ and $S'_i\, \, (1\le i\le 9)$.
None of $S'_i\, \, (1\le i\le 9)$ is endo-commutative.

${\rm{(iii)}}$  $A$  is associative iff it is isomorphic to either one of $ECS^2_7, ECS^2_8, ECS^2_{10}$ and $S'_{10}$.  The algebra $S'_{10}$ is not endo-commutative.
\end{pro}


\vspace{3mm}

\begin{thebibliography}{99}
\bibitem{moduli}
A. Ananin and A. Mironov, The moduli space of 2-dimensional algebras, Comm. Algebra, \textbf{28}(9) (2000), 4481-4488. 

\bibitem{2dim}
M. Goze and E. Remm, 2-dimensional algebras, Afr. J. Math. Phys., \textbf{10}(1) (2011), 81-91. 

\bibitem{JK} 
J. Je\v{z}ek and T. Kepka, Equational theories of medial groupoids, Algebra Universalis {\bf 17}(2) (1983), 174–190.

\bibitem{variety}
I. Kaygorodov and Y. Volkov, The variety of 2-dimensional algebras over an algebraically closed field, Canad. J. Math., \textbf{71}(4) (2019), 819-842. 

\bibitem{level2}
I. Kaygorodov and Y. Volkov, Complete classification of algebras of level two, Mosc. Math. J., \textbf{19}(3) (2019), 485-521. 

\bibitem{2dim-comm-char2} 
Y. Kobayashi, Characterization of two-dimensional commutative algebras over a field of characteristic 2, preprint (2016).

\bibitem{z1} 
Y. Kobayashi, K. Shirayanagi, S.-E. Takahasi and  M. Tsukada, Classification of three-dimensional zeropotent algebras over an algebraically closed field, Comm. Algebra,
  \textbf{45}(12) (2018), 5037-5052.

\bibitem{onkochishin} 
  Y. Kobayashi, K. Shirayanagi, M. Tsukada and S.-E. Takahasi, A complete classification of three-dimensional algebras over $\mathbb R$ and $\mathbb C$ --
  {\it OnkoChishin} (visiting old, learn new), Asian-Eur. J. Math., \textbf{14}(8) (2021), Article ID 2150131 (25 pages). 

\bibitem{medial}
Y. Krasnov and V. Tkachev, Medial and isospectral algebras, arXiv:2210.08245v1 

\bibitem{length1}
O. Markova, C. Mart\'{i}nez and R. Rodrigues, Algebras of length one, J. Pure Appl. Algebra, \textbf{226}(7) (2022), Article ID 106993, 16 p. 

\bibitem{classification}
H. Petersson, The classification of two-dimensional nonassociative algebras, Results Math., \textbf{37}(1-2) (2000), 120-154. 

\bibitem{z2} 
K. Shirayanagi, S.-E. Takahasi, M. Tsukada and Y. Kobayashi, Classification of three-dimensional zeropotent algebras over the real number field, Comm. Algebra, \textbf{46}(11) (2018), 4663-4681.

\bibitem{z3} 
K. Shirayanagi, Y. Kobayashi, S.-E.Takahasi, and M.Tsukada, 
Three-dimensional zeropotent algebras over an algebraically closed field of characteristic two, Comm. Algebra, \textbf{48}(4) (2020), 1613-1625.









\end{thebibliography}
\end{document}